\documentclass[10pt]{amsart}

\usepackage{ mathdots }

\usepackage{psfrag}
\usepackage{amsmath}
\usepackage{amssymb}
\usepackage[english]{babel}
\usepackage{amsthm}
\usepackage{amsfonts}
\usepackage[applemac]{inputenc}
\usepackage{tikz}
\usetikzlibrary{arrows}
\usetikzlibrary{patterns}
\usetikzlibrary{shapes}
\usepackage{mathrsfs}
\usepackage{bbm}

\newtheorem{teo}{Theorem}
\newtheorem{lemma}[teo]{Lemma}

\newtheorem{coro}[teo]{Corollary}

\newcommand{\R}{\mathbb{R}}
\newcommand{\ii}{\mathrm{i}}

\newcommand{\eps}{\epsilon}
\newcommand{\N}{\mathbb{N}}
\newcommand{\C}{\mathbb{C}}
\newcommand{\CP}{\mathbb{C}\textrm{P}}
\newcommand{\RP}{\mathbb{R}\mathrm{P}}

\newcommand{\Z}{\mathbb{Z}}

\newcommand{\Q}{\mathcal{Q}}

\theoremstyle{remark} 
\newtheorem{remark}[]{Remark}

\title{The total Betti number of the intersection of three real quadrics}

\author{ A. Lerario}
\thanks{SISSA, Trieste}

\begin{document}

\maketitle

\begin{abstract}We prove the bound $b(X)\leq n(n+1)$ for the total Betti number of the intersection $X$ of three quadrics in $\RP^{n}$. This bound improves the classical Barvinok's one which is at least of order three in $n$.\end{abstract}

\section{Introduction}
In this paper we address the problem of bounding the total Betti number of the intersection $X$ of three real quadrics in $\RP^{n}$. In the case $X$ were  a \emph{smooth, complete intersection}, then its total Betti number can be easily bounded using Smith's theory: its equations can be \emph{real} perturbed as not to change its topology (it is smooth) and to make its complex points also a smooth complete intersection; in the case of a complete intersection $X_{\C}$ of three quadrics in $\CP^{n}$ it is possible to compute its Betti numbers using Hirzebruch's formula and this would give a bound of the type $b(X)\leq b(X_{\C})\leq p(n)$ where $p$ is a polynomial of degree two.\\
In the general case, i.e. when we make no regularity assumption on $X$, the problem turns out to be more complicated. If we simply naively perturb the equations defining $X$ we can of course make its complex point to be smooth, but then the topology of the real part would have change.\\
The very first attempt to bound the topology of the intersection of $k$ quadrics in $\RP^{n}$ is to use the well known
Oleinik-Petrovskii-Thom-Milnor inequality (see \cite{BPR}), which gives the
estimate\footnote{According to \cite{BPR} in this context the notation $f(n)=O(n)$ means that there exists a natural number $b$ such that the inequality $f(n)\leq bn$ hods for every $n\in \N$.}:
$$b(X)\leq O(k)^{n}.$$
Surprisingly enough it turns out that the fact that $X$ is
defined by quadratic equations allows to interchange the role of the two numbers
$n$ and $k$ and to get the classical Barvinok's bound (see \cite{Ba}): 
$$b(X)\leq n^{O(k)}.$$ The hidden constant in the exponent for this estimate is at least two, as noticed also by the authors of \cite{BaKe}, where a more refined estimate is presented (but of the same leading order).\\
In particular in the case $X$ is the intersection of \emph{three} quadrics in $\RP^{n}$ this classical estimates would give 
$$b(X)\leq n^{O(3)}.$$
The passage from Oleinik-Petrovskii-Thom-Milnor bound to Barvinok's one is essentially made using a kind of duality argument, which works for the quadratic case, between the number of variables and the number of equations. This idea appeared for the first time in the paper \cite{Agrachev2} and we explain it now. If we have the quadratic forms $q_{0},\ldots, q_{k}$ on $\R^{n+1},$ then we can consider their linear span $L$ in the space of all quadratic forms. The arrangement of $L$ with respect to the subset $Z$ of degenerate quadratic forms (those with at least one dimensional kernel) determines in a very precise way the topology of the base locus: 
$$X=\bigcap_{q\in L\backslash \{0\}}\{[x]\in \RP^{n}\, |\,q(x)\leq 0\}.$$ The simplest invariant we can associate to a quadratic form $q$ is its positive inertia index $\ii^{+}(q)$, namely the maximal dimension of a subspace $V\subset \R^{n+1}$ such that $q|_{V}$ is positive definite. In a similar fashion we are led to consider for $j\in \N$ the sets:
	$$\Omega^{j}=\{q\in L\backslash \{0\}\,|\, \ii^{+}(q)\geq j\}.$$ The spirit of the mentioned duality is in this procedure of replacing the original framework with the filtration:
	$$\Omega^{n+1}\subseteq\Omega^{n}\subseteq\cdots\subseteq \Omega^{1}\subseteq\Omega^{0}.$$  This duality is widely investigated in the paper \cite{AgLe}, where a spectral sequence converging to the homology of $X$ is studied; this spectral sequence has second term $E_{2}^{i,j}$ isomorphic to $H^{n-i}(L,\Omega^{j+1}).$ Thus, at a first approximation, the previous cohomology groups can be taken as the homology of $X$ and as long as we consider finer properties of the arrangement of $L$, then new information on the topology of its base locus is obtained. It is remarkable that only using this approximation the classical Barvinok's bounds can be recovered; in this setting they can be formulated in the form:
\begin{equation}\label{eq:top}b(X)\leq n+1+\sum_{j\geq 0}b(\Omega^{j+1})\end{equation}
The introduction of the full methods from \cite{AgLe} made also possible in some cases to strongly improve the classical bounds. For example in the paper \cite{Le} the author proves that the total Betti number of the intersection $X$ of two quadrics in $\RP^{n}$ is bounded by $3n+2$; in the same papers are provided also bounds linear in $k$ for each specific Betti number of $X.$\\
The intersection of $L$ with the set of degenerate forms $Z$ is customary called the spectral variety $C$ of the base locus $X$. In this paper we present the idea that in the case $X$ is the intersection of \emph{three} quadrics, then its topological complexity is essentially that of its spectral variety $C;$ this variety is in fact the ``difference'' between the various sets $\Omega^{j}$  and thus the sum in the right hand side of (\ref{eq:top}) in a certain sense is ``bounded'' by $b(C).$ This idea originally appears in \cite{Agrachev1} in the regular case. Here enters the deep connection, generally called Dixon's correspondence, between the intersection of three real quadrics in $\RP^{n}$ and curves of degree $n+1$ on $\RP^{2}$: in the case $X$ is a \emph{smooth, complete intersection} the corresponding curve is the projectivization of the spectral variety (see \cite{Dixon} and \cite{DeItKh}). In our framework $X$ is no longer smooth, nor a complete intersection, and a pertubative approach is introduced to study it; this approach associates to $X$ a \emph{smooth} curve  in $S^{2}$ which replaces the role of the spectral variety (indeed in the regular case this curve is the double cover of the curve given by Dixon's correspondence). We relate the complexity of $X$ to that of this new curve and using a Harnack's type argument on the sphere will give us the mentioned bound $b(X)\leq n(n+1).$
\section*{Acknowledgements}
The author is grateful to his teacher A. A. Agrachev: the main idea of this paper is a natural development of his intuition.

\section{General settings}

We recall in this section a general construction to study the topology of intersection of real quadrics. We set $\Q(n+1)$ for the space of real quadratic forms in $n+1$ variables; if $q_{0},\ldots, q_{k}$ belong to this space, then we can consider their common zero locus $X$ in $\RP^{n}:$
$$X=V_{\RP^{n}}(q_{0},\ldots, q_{k}), \quad q_{0},\ldots, q_{k}\in \Q(n+1)$$
To study the topology of $X$ we introduce the following auxiliary construction. We denote by $q$ the $(k+1)$-ple $(q_{0},\ldots, q_{k})$ and consider the map $\overline q:S^{k}\to \Q(n+1)$, defined by
$$\omega \mapsto \omega q=\omega_{0}q_{0}+\cdots+\omega_{k}q_{k},\quad \omega=(\omega_{0},\ldots, \omega_{k})\in S^{k}.$$
This map places the unit sphere $S^{k}$ linearly into the space of quadratic forms, in the direction of the chosen quadrics. For a given quadratic form $p\in Q(n+1)$ we denote by $\ii^{+}(p)$ its positive inertia index, namely the maximal dimension of a subspace of $\R^{n+1}$ such that the restriction of $p$ to it is positive definite. For a \emph{family} of quadratic forms depending on some parameters, like the map $\overline q$ describes, we consider the geometry of this function on the parameter space. We are thus naturally led to define the sets:
$$\Omega^{j}=\{\omega \in S^{k}\, |\, \ii^{+}(\omega q)\geq j\},\quad j\in \N$$
The following theorem relates the topological complexity of $X$ to that of the sets $\Omega^{j}.$ For a semialgebraic set $S$ we define $b(S)$ to be the sum of its Betti numbers\footnote{From now on every homology group is assumed with $\Z_{2}$ coefficients, unless differently stated; the same remark applies to Betti numbers.}.
 \begin{teo}[Topological complexity formula]
$$b(X)\leq n+1+\sum_{j\geq0}b(\Omega^{j+1})$$
\begin{proof}Consider the topological space $B=\{(\omega,[x])\in S^{k}\times \RP^{n}\,|\,  (\omega q)(x)>0\}$ together with its two projections $p_{1}:B\to S^{k}$ and $p_{2}:B\to \RP^{n}.$ The image of $p_{2}$ is easily seen to be $\RP^{n}\backslash X$ and the fibers of this map are contractible sets, hence $p_{2}$ gives a homotopy equivalence $B\sim\RP^{n}\backslash X.$ Consider now the projection $p_{1};$ for a point $\omega \in S^{2}$ the fiber $p_{1}^{-1}(\omega)$ has the homotopy type of a projective space of dimension $\ii^{+}(\omega q)-1$, thus the Leray spectral sequence for $p_{1}$ converges to $H^{*}(\RP^{n}\backslash X)$ and has the second terms $E_{2}^{i,j}$ isomorphic to $H^{i}(\Omega^{j+1}).$ A detailed proof of the previous statements can be found in \cite{AgLe}. Since $\textrm{rk}(E_{\infty})\leq \textrm{rk}(E_{2})$ then $b(\RP^{n}\backslash X)\leq \sum_{j\geq0} b(\Omega^{j+1})$. Recalling that by Alexander-Pontryagin duality $H_{n-*}(X)\simeq H^{*}(\RP^{n},\RP^{n}\backslash X),$ then the exactness of the long cohomology exact sequence of the pair $(\RP^{n},\RP^{n}\backslash X)$ gives the desired inequality. \end{proof}
\end{teo}

\begin{remark}A more refined formula for $b(X)$ follows by considering a different spectral sequence directly converging to $H_{n-*}(X)$. In fact by \cite{AgLe} there exists such a spectral sequence $(E_{r},d_{r})_{r\geq 0}$ with second term $E_{2}^{i,j}\simeq H^{i}(B,\Omega^{j+1})$, where $B$ is the unit ball in $\R^{k+1}$ and $\Omega^{j+1}\subseteq \partial B.$ If we let $\mu$ be the maximum of $\ii^{+}$ on $S^{k}$ and $\nu$ be its minumum, then we get $b(X)\leq \textrm{rk}(E_{2})\leq n+1-2(\mu-\nu)+\sum_{\nu+1\leq j+1\leq \mu}b(\Omega^{j+1})$. The paper \cite{AgLe} contains a description of the second differential of this spectral sequence, which happens to be related with the set of points on $S^{k}$ where $\omega q$ has multiple eigenvalues, together with applications.
\end{remark}

\begin{remark}By universal coefficients theorem, the previous bound is valid also for the total Betti number of $X$ with coefficients in $\Z$ (but on the right hand side $\Z_{2}$ coefficients are still assumed).
\end{remark}

The previous formula, together with some results from \cite{BPR}, can be used to give the classical Barvinok's estimate (see the paper \cite{Ba}).
\begin{coro}[Barvinok's estimate]$$b(X)\leq (n+1)^{O(2k+2)}.$$
\begin{proof}Let us fix a scalar product; then the rule $\langle x,(\omega Q)x\rangle =q(x)$ defines a symmetric matrix $\omega Q$ whose number of positive eigenvalues equals $\ii^{+}(\omega q).$ Consider the polynomial $\det (\omega Q-tI)=a_{0}(\omega)+\cdots+a_{n}(\omega)t^{n}\pm t^{n+1};$ then by Descartes' rule of signs  the positive inertia index of $\omega Q$ is given by the sign variation in the sequence $(a_{0}(\omega), \ldots, a_{n}(\omega)).$ Thus the sets $\Omega^{j+1}$ are defined on the sphere $S^{k}$ by sign conditions (quantifier-free formulas)  whose atoms are polynomials in $k+1$ variables and of degree less than $n+1$. For such sets we have the estimate, proved in \cite{BPR}: $b(\Omega^{j+1})\leq (n+1)^{O(2k+1)}$. Putting all them together we get:
	$$b(X)\leq n+1+\sum_{j=0}^{n}b(\Omega^{j+1})\leq (n+1)^{O(2(k+1))}$$
(notice that $k+1$ is the number of quadrics cutting $X$).
\end{proof}
\end{coro}

\begin{remark}The paper \cite{Ba} contains the bound for the set $S$ of solutions of $k$ quadratic inequalities in $\R^{n+1},$ which is $b(S)\leq (n+1)^{O(k)}$. The set $X$ in $\RP^{n}$ can be viewed as double covered by a subset $X'$ in $\R^{n+1}$ defined by $k+2$ quadratic inequalities (those defining $X$ together with the quadratic equation for the unit sphere); by the transfer exact sequence $b(X)\leq \frac{1}{2}b(X')$ and by Barvinok's estimate $b(X')\leq (n+1)^{O(k+2)}$; since the constant hidden in the previous exponent is at least two, then we get the same order as in the previous corollary.\\
A more refined bound in the general case can be found in \cite{BaKe}.
\end{remark}

\section{Preliminaries on perturbations}
Let now $p\in \Q(n+1)$ be a positive definite form (we will use the notation $p>0$ for such forms) and for every $\eps>0$ and $k\in \N$ let us define the sets:
$$\Omega_{n-j}(\eps)=\{\omega \in S^{k}\, |\, \ii^{-}(\omega q-\eps p)\leq n-j\}$$
where $\ii^{-}$ denotes the negative inertia index, i.e. $\ii^{-}(\omega q-\eps p)=\ii^{+}(\eps p-\omega q).$ The following lemma relates the topology of $\Omega^{j+1}$ and of its perturbation $\Omega_{n-j}(\eps).$ 
\begin{lemma}
\label{union}For every positive definite form $p\in \Q(n+1)$  and for every $\eps>0$ sufficiently small
$$b(\Omega^{j+1})=b(\Omega_{n-j}(\eps)).$$
\begin{proof}
Let us first prove that $\Omega^{j+1}=\bigcup_{\eps>0}\Omega_{n-j}(\eps).$\\
Let $\omega \in \bigcup_{\eps>0}\Omega_{n-j}(\eps);$ then there exists $\overline{\eps}$ such that $\omega\in \Omega_{n-j}(\eps)$ for every $\eps<\overline{\eps}.$ Since for $\eps$ small enough $$\ii^{-}(\omega q-\eps p)=\ii^{-}(\omega q)+\dim (\ker (\omega q))$$ then it follows that $$\ii^{+}(\omega q)=n+1-\ii^{-}(\omega q)-\dim (\ker \omega q)\geq j+1.$$ Viceversa if $\omega \in \Omega^{j+1}$ the previous inequality proves $\omega\in \Omega_{n-j}(\eps)$ for $\eps$ small enough, i.e. $\omega \in \bigcup_{\eps>0}\Omega_{n-j}(\eps).$\\
Notice now that if $\omega\in \Omega_{n-j}(\eps)$ then, eventually choosing a smaller $\eps$, we may assume $\eps$ properly separates the spectrum of $\omega$ and thus, by continuity of the map $\overline q$, there exists $U$ open neighborhood of $\omega$ such that $\eps$ properly separates also the spectrum of $\eta q$ for every $\eta \in U$ (see \cite{Kato} for a detailed discussion of the regularity of the eigenvalues of a family of symmetric matrices). Hence every $\eta \in U$ also belongs to $\Omega_{n-j}(\eps)$. From this consideration it easily follows that each compact set in $\Omega^{j+1}$ is contained in some $\Omega_{n-j}(\eps)$ and thus $$\varinjlim_{\eps}\{H_{*}(\Omega_{n-j}(\eps))\}=H_{*}(\Omega^{j+1}).$$ It remains to prove that the topology of $\Omega_{n-j}(\eps)$ is definitely stable in $\eps$ going to zero. Consider the semialgebraic compact set $S_{n-j}=\{(\omega, \eps)\in S^{k}\times [0, \infty)\, |\, \ii^{-}(\omega q-\eps p)\leq n-j\}$. By Hardt's triviality theorem (see \cite{BCR}) we have that the projection $(\omega, \eps)\mapsto \omega$ is a locally trivial fibration over $(0,\eps)$ for $\eps$ small enough; from this the conclusion follows.
\end{proof}
\end{lemma}

Let us now move to specific properties of pencils of three quadrics.\\
We recall that the space of degenerate forms $Z=\{p\in \Q(n+1)\, |\, \ker(p)\neq0\}$ admits a semialgebraic Nash stratification $Z=\coprod Z_{i}$ such that its singularities belong to strata of codimension at least three in $Q(n+1);$ this is a classical result and the reader can see \cite{Agrachev2} for a direct proof.

\begin{lemma}\label{perturb}There exists a positive definite form  $p\in \Q(n+1)$ such that for every $\eps>0$ small enough the curve 
$$C(\eps)=\{\omega \in S^{2}\, |\, \textrm{ker}(\omega q-\eps p)\neq 0\}$$
is a smooth curve in $S^{2}$ such that the difference of the index function $\omega \mapsto \ii^{-}(\omega q-\eps p)$ on adjacent components of $S^{2}\backslash C(\eps)$ is $\pm 1.$
\end{lemma}
\begin{proof}Let $\Q^{+}$ be set of positive definite quadratic forms in $\Q(n+1)$ and consider the map $F:S^{2}\times \Q^{+}$ defined by 
$$(\omega, p)\mapsto \omega q-p.$$
Since $\Q^{+}$ is open in $\Q,$ then $F$ is a submersion and $F^{-1}(Z)$ is Nash-stratified by $\coprod F^{-1}(Z_{i}).$ Then for $p\in \Q^{+}$ the evaluation map $\omega \mapsto f(\omega)-p$ is transversal to all strata of $Z$ if and only if $p$ is a regular value for the restriction of the second factor projection $\pi:S^{2}\times \Q^{+}\to \Q^{+}$ to each stratum of $F^{-1}(Z)=\coprod F^{-1}(Z_{i}).$
Thus let $\pi_{i}=\pi|_{F^{-1}(Z_{i})}:F^{-1}(Z_{i})\to \Q^{+};$ since all datas are smooth semialgebraic, then by semialgebraic Sard's Lemma (see \cite{BCR}), the set $\Sigma_{i}=\{\hat{q}\in \Q^{+}\, | \, \hat{q}\textrm{ is a critical value of $\pi_{i}$}\}$ is a semialgebraic subset of $\Q^{+}$ of dimension strictly less than $\dim (\Q^{+}).$ Hence $\Sigma=\cup_{i}\Sigma_{i}$ also is a semialgebraic subset of $\Q^{+}$ of dimension $\dim (\Sigma)<\dim (\Q^{+})$ and for every $p\in \Q^{+}\backslash \Sigma$ the map $\omega\mapsto f(\omega)-p$ is transversal to each $Z_{i}.$ Since $\Sigma$ is semialgebraic of codimension at least one, then there exists $p\in \Q^{+}\backslash \Sigma$ such that $\{t p\}_{t>0}$ intersects $\Sigma$ in a finite number of points, i.e. for every $\eps>0$ sufficiently small $\eps p\in \Q^{+}\backslash \Sigma.$ Since the codimension of the singularities of $Z$ are at least three, then for $p\in \Q^{+}\backslash \Sigma$ and $\eps>0$ small enough the curve $\{\omega \in S^{2}\, |\,\ker(\omega q-\eps p)\neq0\}$ is smooth. Moreover if $z$ is a smooth point of $Z$, then its normal bundle at $z$ coincides with the sets of quadratic forms $\{\lambda (x\otimes x)\,|\, x\in \ker(z)\}_{\lambda \in \R}$ then also the second part of the statement follows.\end{proof}

Essentially lemma \ref{perturb} tells that we can perturb the map $\omega\mapsto \omega q$ in such a way that the set of points where the index function can change is a \emph{smooth} curve on $S^{2};$ lemma \ref{union} tells us how to control the topology of the sets $\Omega^{j+1}$ after this perturbation.

\begin{remark}In higher dimension all that we can do is perturb the map $\overline q$ as to make it as best as possible, i.e. transversal to all strata of $Z=\coprod Z_{i};$ for example in the case of four quadrics we can make the hypersurface $\{\omega\in S^{3}\, |\,\ker(\omega q-\eps p)\neq 0\}$ a real algebraic manifold of dimension two with at most isolated singularities.
\end{remark}

\section{Harnack's type inequalities}
In this section we bound the topology of a smooth curve $C\subset S^{2}$ (as the above one) in a way similar to Harnack's classical bounds for smooth curves in $\RP^{2}.$
We start with the following lemma. 
\begin{lemma}\label{polperturb}Let $G\in \R[\omega_{0},\ldots, \omega_{3}]$ be a homogeneous polynomial. The set $\Sigma$ of all homogeneous polynomials $F$ of a fixed degree $d$ such that $V_{\CP^{3}}(F,G)$ is not a smooth complete intersection in $\CP^{3}$ is a proper real algebraic set in $\R[\omega_{0},\ldots, \omega_{3}]_{(d)}$.
\begin{proof}The set $\Sigma_{\C}$ of homegeneous polynomials $H$ with complex coefficients and degree $d$ such that $V_{\CP^{3}}(H,G)$ is not a complete intersection in $\CP^{3}$ is clearly a proper complex algebraic subset of $\C[\omega_{0},\ldots, \omega_{3}]_{(d)}\simeq \C^{N},$ where $N=\left(\begin{smallmatrix}4+d\\d\end{smallmatrix}\right)$. Let $\Sigma_{\C}$ be defined by the vanishing of certain polynomials $f_{1},\ldots, f_{l}$ in $\C[z_{1},\ldots, z_{N}]$. Since $\Sigma$ equals $\Sigma_{\C}\cap \R^{N},$ then it is real algebraic; it remains to show it is  \emph{proper}. Suppose not. Then $\Sigma_{\C}\cap \R^{N}=\R^{N}$, which means that each of the $f_{i}$ vanishes identically over the reals. In particular, fixing all but one variables in $f_{i}$ we have a complex polynomial in one variable which has infinite zeroes, hence it must be zero. Iterating the reasoning for each variable this would give that each $f_{i}$ is zero, which is absurd since $\Sigma_{\C}$ was a proper algebraic set.
\end{proof}\end{lemma}
\begin{lemma}\label{Harnack}Let $f\in\R[\omega_{0},\omega_{1},\omega_{2}]$ be a polynomial of degree $d$ such that 
$$C=\{(\omega_{0},\omega_{1},\omega_{2}) \in S^{2}\, |\, f(\omega_{0},\omega_{1},\omega_{2})=0\}$$
is a smooth curve; then the number of its ovals is at most $d(d-2)+2$.
\begin{proof}If $f$ is homogeneous, then $C$ is the double cover of a smooth curve $C'$ in $\RP^{2}$ of degree $d;$ hence by Harnack's inequality $b(C')\leq (d-1)(d-2)+2.$ By the transfer exact sequence $b(C)\leq 2b(C')$, which in this case gives the bound $(d-1)(d-2)+2\leq d(d-2)+2$ for the number of the ovals of $C$.\\
Assume now $f$ is not homogeneous and let $F\in \R[\omega_{0},\dots,\omega_{3}]$ be its homogenization (the new variable is $\omega_{3}$); let also $G$ be the polynomial $G(\omega_{0},\ldots, \omega_{3})=\omega_{0}^{2}+\omega_{1}^{2}+\omega_{2}^{2}-\omega_{3}^{2}.$ Using this setting we have that the curve $C$ coincides with:
$$V_{\RP^{3}}(F,G)\subset \RP^{3}$$
(there are no solutions on the hyperplane $\{\omega_{3}=0\}$ to $F=G=0$). By lemma \ref{polperturb} there exists a \emph{real} perturbation $F_{\eps}$ of the polynomial $F$, homogeneous and of the same degree of $F$, such that
$$V_{\CP^{3}}(F_{\eps},G)\subset \CP^{3}$$
is a smooth complete intersection in $\CP^{3}$. Moreover since the perturbation was real, then by Smith's theory the total Betti number of $V_{\RP^{3}}(F_{\eps},G )$ is bounded by that of $V_{\CP^{3}}(F_{\eps},G);$ on the other hand since $V_{\RP^{3}}(F,G)$ was smooth, then a small perturbation of its equations does not change its topology, hence the total Betti number of $C=V_{\RP^{3}}(F,G)$ also is bounded by that of $V_{\CP^{3}}(F_{\eps},G).$ It remains to prove that for the complete intersection $V_{\CP^{3}}(F_{\eps},G)$ the bound on its topological complexity is $2d(d-2)+4$. To this end notice that $Y=V_{\CP^{3}}(F_{\eps},G)$ is a smooth complex curve of degree $2d;$ hence if we let $K_{Y}$ be its canonical bundle the adjunction formula reads $K_{Y}=O_{\CP^{3}}(d-2)|_{Y}=(O_{\CP^{3}}(1)|_{Y})^{\otimes (d-2)}$. Since the degree of $K_{Y}$ is $2g(Y)-2$, then
$$2g(Y)-2=(d-2)\textrm{deg}(O_{\CP^{3}}(1)|_{Y})=(d-2)2d.$$
Since $b(Y)=2g(Y)+2$ this concludes the proof. \end{proof}
\end{lemma}

\begin{coro}\label{curve}There exists a positive definite form $p$ in $\Q(n+1)$ such that for every $\eps>0$ small enough the smooth curve 
$$C(\eps)=\{\omega \in S^{2}\, |\, \textrm{ker}(\omega q-\eps p)\neq 0\}$$ has at most $(n+1)(n-1)+2$ ovals.
\begin{proof}Let $p$ be given by lemma \ref{perturb}. Fix a scalar product in such a way that each quadratic form can be identified with a real symmetric matrix, as in the proof of Barvinok's estimate. Thus $\omega Q$ and $P$ are the symmetric matrices associated respectively to $\omega q$ and $p$. The conclusion follows simply by applying the previous lemma to the polynomial $f(\omega_{0},\omega_{1},\omega_{2})=\textrm{det}(\omega Q-\eps P)$, which has degree $n+1$.
\end{proof}
\end{coro}

We recall in this section also the following elementary fact.
\begin{lemma}\label{surfaces}Let $\Omega\subset S^{2}$ be a surface with boundary $\partial \Omega\neq \emptyset$. Then:
$$b(\Omega)=b_{0}(\partial \Omega)$$
\end{lemma}
\begin{proof}By additivity of the formula it is sufficient to prove it in the case $\Omega$ is connected. In this case $\Omega$ is homotopy equivalent to the sphere $S^{2}$ minus $b_{0}(\partial \Omega)$ points and thus $b_{0}(\Omega)=1$ and $b_{1}(\Omega)=b_{0}(\Omega)-1.$
\end{proof}

\section{The total Betti number of the intersection of three real quadrics}
The aim of this section is to prove the following theorem, which estimates the total Betti number of the intersection $X$ of three real quadrics in $\RP^{n}.$ Notice that we do not require any nondegeneracy assumption.
\begin{teo}Let $X$ be the intersection of three quadrics in $\RP^{n}.$ Then: $$b(X)\leq n(n+1)$$
\begin{proof}We use the refined formula $b(X)\leq n+1-2(\mu-\nu)+\sum_{\nu+1\leq j+1\leq \mu}b(\Omega^{j+1})$ for the total Betti number of $X$; if we use only the topological complexity formula the estimate we can produce is a bit worst, but always of the type $n^{2}+ \textrm{terms of lower degree}$). By lemma \ref{perturb} there exists a positive definite form $p$ such that for every $\eps>0$ sufficiently small the curve $C(\eps)=\{\omega \in S^{2}\, |\, \textrm{ker}(\omega q-\eps p)\neq 0\}$ is smooth; moreover by lemma \ref{union} for such a $p$ and for $\eps>0$ small enough we also have the equality $b(\Omega^{j+1})=b(\Omega_{n-j}(\eps))$. This in particular gives
$$b(X)\leq n+1-2(\mu-\nu)+\sum_{\nu+1\leq j+1\leq \mu}b(\Omega_{n-j}(\eps)).$$
Since for each $\nu+1\leq j+1\leq \mu$ the set $\Omega_{n-j}(\eps)$ is a submanifold of $S^{2}$ with nonempty boundary, then by lemma \ref{surfaces}:
$$b(\Omega_{n-j}(\eps))=b_{0}(\partial \Omega_{n-j}(\eps)).$$
In particular $\sum b(\Omega_{n-j}(\eps))$ equals $\sum b_{0}(\partial \Omega_{n-j}(\eps)),$ where in both cases the sum is made over the indexes $\nu+1\leq j+1\leq \mu$. The second part of lemma \ref{perturb} implies now that each of the ovals of $C(\eps)$ belongs to the boundary of exactly one of the $\Omega_{n-j}(\eps), \nu+1\leq j+1\leq \mu.$ This implies that the previous sum $\sum b(\partial \Omega_{n-j}(\eps))$ equals exactly the number $c$ of ovals of $C(\eps).$ In particular this gives:
$$b(X)\leq n+1-2(\mu-\nu)+c.$$
If $\mu=\nu,$ then $b(X)\leq n+1$; thus we may assume $2(\mu-\nu)\geq 2$. Corollary \ref{curve} tells that $c\leq (n+1)(n-1)+2$, which finally gives:
$$b(X)\leq n+1-2+(n+1)(n-1)+2=n(n+1).$$
 \end{proof}
\end{teo}

\begin{remark}Since in the previous proof the sets $\Omega_{n-j}(\eps)$ and their boundaries were semialgebraic subsets of $S^{2},$ then their Betti numbers with coefficients in $\Z$ coincide with those with coefficient in $\Z_{2};$ moreover by the universal coefficient theorem $b(X;\Z)\leq b(X)$ and thus we also have:
$$b(X;\Z)\leq n(n+1).$$
\end{remark}
  \begin{remark}
If we define the map $q:\R^{n+1}\to \R^{3}$ whose components are $(q_{0},q_{1},q_{2})$, then the intersection of the three quadrics defined by the vanishing of $q_{0},q_{1}$ and $q_{2}$ equals $\{[x]\in \RP^{n}\, |\, q(x)=0\}.$ In a similar way if $K\subset \R^{3}$ is a closed polyhedral cone, we may define (by slightly abusing of notations) the set $q^{-1}(K)=\{[x]\in \RP^{n}\, |\, q(x)\in K\}$. Such a set is the set of the solutions of a system of three quadratic inequalities in $\RP^{n}$ and using the spectral sequence of \cite{AgLe} and a similar argument as above it is possible to prove a bound quadratic in $n$ for its topological complexity. We leave the details to the reader.
\end{remark}

\begin{remark}
In the case $X$ is a \emph{complete intersection} of quadrics in $\RP^{n}$, estimates on the number of its connected components are given in \cite{DeItKh}. In particular, following the notations of \cite{DeItKh}, we can denote by $B_{r}^{k}(n)$ the maximum value that the $k$-th Betti number of an intersection (not necessarily complete) of $r+1$ quadrics in $\RP^{n}$ can have. There it is proved that for complete intersections
$$B_{2}^{0}(n)\leq \frac{3}{2}l(l-1)+2,\quad l=[n/2]+1.$$
The reader should notice that, in accordance with our result, the estimate is quadratic in $n$; our bound tells in particular that this quadratic estimate holds for \emph{every} Betti number and also without regularity assumptions.
\end{remark}

\end{document}